\documentclass[twoside,12pt, a4paper]{amsart}
\usepackage{amssymb,supertabular,a4}
\usepackage{amsmath,amsthm,amscd,amssymb,graphicx}
\input xy
\xyoption{all}
\def\Hom{\hbox{\rm Hom\,}}

\def\Hom{\hbox{\rm Hom\,}}

\def\im{\hbox{\rm im}}

\def\mod{\hbox{\rm mod\,}}
\def\Rep{\hbox{\rm Rep\,}}

\def\sing{\hbox{\rm Sing}}
\def\udim{\hbox{\rm \underline{dim}\,}}
\def\ra{{\rightarrow\,}}

\title{The Gabriel-Roiter measure for  $\widetilde{\mathbb{A}}_n$ II}
\author{Bo Chen}
\date{}
\email {mcebbchen@googlemail.com}
\address {Universit\"at zu K\"oln\\
          Mathematisches Institut\\
                   Weyertal 86-90\\
           D-50931 K\"oln\\ Germany}

\thanks {The author is supported by DFG-Schwerpunktprogramm
`Representation theory'.}


\newtheorem{theo}{Theorem}[section]

\newtheorem{lemm}[theo]{Lemma}
\newtheorem{coro}[theo]{Corollary}
\newtheorem{prop}[theo]{Proposition}
\newtheorem*{theo*}{Theorem A}
\newtheorem*{theo**}{Theorem B}

\newenvironment{remark}[1][Remark]{\begin{trivlist}
\item[\hskip \labelsep {\bfseries #1}]}{\end{trivlist}}

\begin{document}

\begin{abstract}
Let $Q$ be a tame quiver of type $\widetilde{\mathbb{A}}_n$ and
$\Rep(Q)$ the category of finite dimensional  representations over
an algebraically closed field. A representation is simply called a
module.  It will be shown that a regular string module has, up to isomorphism,
at most two Gabriel-Roiter submodules.
The quivers $Q$ with sink-source orientations will be characterized as
those, whose central parts do not contain
preinjective modules.
It will also be shown
that there are only finitely many (central) Gabriel-Roiter measures admitting no
direct predecessors. This fact will be generalized for all tame quivers.

\end{abstract}

\maketitle


\section{Introduction}

Let $\Lambda$ be an Artin algebra and $\mod\Lambda$ the category of
finitely generated left $\Lambda$-modules. For each
$M\in\mod\Lambda$, we denote by $|M|$ the length of $M$. The symbol
$\subset$ is used to denote proper inclusion. The Gabriel-Roiter (GR
for short) measure $\mu(M)$ for a $\Lambda$-module $M$ was defined to be a
rational number
in \cite{R2} by induction as follows:
$$\mu(M)=\left\{\begin{array}{ll}
0 & \textrm{if $M=0$};\\
\max_{N\subset M}\{\mu(N)\} & \textrm{if $M$ is decomposable};\\
\max_{N\subset M}\{\mu(N)\}+\frac{1}{2^{|M|}} & \textrm{if $M$ is indecomposable}.\\
\end{array}\right.$$
(In later discussion, we will use the original definition for our
convenience, see \cite{R1} or Section \ref{GRdef} below.) The
so-called Gabriel-Roiter submodules of an indecomposable module are
defined to be the indecomposable proper submodules with maximal GR
measure.

A rational number $\mu$ is called a GR measure for an algebra $\Lambda$ if
there is an indecomposable $\Lambda$-module $M$ such that $\mu(M)=\mu$.
Using the Gabriel-Roiter  measure, Ringel obtained a partition of the
module category for any Artin algebra of infinite representation
type \cite{R1,R2}: there are infinitely many GR measures $\mu_i$ and
$\mu^i$ with $i$ natural numbers, such that
$$\mu_1<\mu_2<\mu_3<\ldots\quad \ldots <\mu^3<\mu^2<\mu^1$$ and such that any
other GR measure $\mu$ satisfies $\mu_i<\mu<\mu^j$ for all $i,j$.
The GR measures $\mu_i$ (resp. $\mu^i$) are called take-off (resp.
landing) measures. Any other GR measure is called a central measure.
An indecomposable module  is  called a  take-off (resp. central,
landing) module if its GR measure is a take-off (resp. central,
landing) measure.

To calculate the GR measure of a given indecomposable module, it is
necessary to know the GR submodules. Thus it is interesting to know
the number of the isomorphism classes of the GR submodules for a
given indecomposable module. It was conjectured that for a
representation-finite algebra (over an algebraically closed field),
each indecomposable module has at most three GR submodules.  In
\cite{Ch1}, we have proved the conjecture for representation-finite
hereditary algebras. In this paper, we will start to study the GR
submodules of string modules. In particular, we will show in Section
\ref{string}
that each string module, which contains no band submodules, has at
most two GR submodules, up to isomorphism. As an application, we
show for a tame quiver $Q$ of type $\widetilde{\mathbb{A}}_n$, $n\geq 1$, that
if an indecomposable module on an exceptional regular component of the Auslander-Reiten quiver
has, up to isomorphism, precisely two GR
submodules, then one of the two GR inclusions is an irreducible
monomorphism. A description of the numbers of the GR submodules will
also be presented there.

Let $\mu,\mu'$ be two GR measures for $\Lambda$. We call $\mu'$  a
{\bf direct successor} of $\mu$ if, first, $\mu<\mu'$ and second,
there does not exist a GR measure $\mu''$ with $\mu<\mu''<\mu'$. The
so-called {\bf Successor Lemma} in \cite{R2} states that any
Gabriel-Roiter measure $\mu$ different from $\mu^1$ has a direct
successor. There is no `Predecessor Lemma'.  It is
clear that any GR measure different from $\mu_1$, the minimal one,
over a representation-finite Artin algebra
has a direct predecessor.  In this paper, we start to study the GR measures 
admitting no direct predecessors
for representation-infinite algebras.
An ideal test class are the path algebras of tame quivers, as the representation
theory of this category is very well understood and many first properties of the GR-measures
for these algebras are already known
\cite{Ch4}. Among them, the quivers of type
$\widetilde{\mathbb{A}}_n$, $n\geq 1$,
are of special interests because the GR submodules of an indecomposable module
can be in some sense easily determined (see \cite{Ch3}, or Proposition \ref{bigprop} below).
In Section \ref{predecessors}, it will be shown that for a quiver of type $\widetilde{\mathbb{A}}_n$
there are only finitely many GR measures admitting no direct predecessors.
This gives rise to the following question:  does a
representation-infinite (hereditary) algebra (over an algebraically
closed field) being of tame type imply that there are only finitely
many GR measures having no direct predecessors and vice versa?
The proof of the above mentioned result for quivers of
type $\widetilde{\mathbb{A}}_n$ can be generalized to those of types
$\widetilde{\mathbb{D}}_n$ ($n\geq 4$), $\widetilde{\mathbb{E}}_6$,
$\widetilde{\mathbb{E}}_7$ and $\widetilde{\mathbb{E}}_8$.
Thus we partially answer the above question.

It was shown in \cite{R1} that all landing modules are preinjective
modules in the sense of Auslander-Smal\o\ \cite{AS}. However, not
all preinjective modules are landing modules in general. It is
interesting to study the preinjective modules, which are in the central
part. In Section \ref{precen}, We will show that for a tame quiver $Q$ of type
$\widetilde{\mathbb{A}}_n$, if there is a preinjective central
module, then there are actually infinitely many of them. However, it is
possible that the central part does not contain any preinjective
module. We characterize the tame quivers of type
$\widetilde{\mathbb{A}}_n$ with this property. In particular, we
show that the quiver $Q$ of type $\widetilde{\mathbb{A}}_n$ is
equipped with a sink-source orientation if and only if any
indecomposable preinjective module is either a landing module or a
take-off module.

Throughout, we  fix an algebraically closed field $k$ and by algebras, representations or modules
we always mean finite dimensional ones over $k$, unless stated otherwise.

\section{Preliminaries and known results}\label{preliminaries}
\subsection{The Gabriel-Roiter measure}\label{GRdef}
We first  recall the original definition of the Gabriel-Roiter measure
\cite{R1}. Let $\mathbb{N}$=$\{1,2,\ldots\}$ be the set of
natural numbers and $\mathcal{P}(\mathbb{N})$ be the set of all
subsets of $\mathbb{N}$.  A total order on
$\mathcal{P}(\mathbb{N})$ can be defined as follows: if $I$,$J$
are two different subsets of $\mathbb{N}$, write $I<J$ if the
smallest element in $(I\backslash J)\cup (J\backslash I)$ belongs to
J. Also we write $I\ll J$ provided $I\subset J$ and for all elements
$a\in I$, $b\in J\backslash I$, we have $a<b$. We say that $J$ {\bf
starts with} $I$ if $I=J$ or $I\ll J$.

Let $\Lambda$ be an Artin algebra and  $\mod\Lambda$ be the category
of finitely generated left $\Lambda$-modules. For each
$M\in\mod\Lambda$, let $\mu(M)$ be the maximum of the sets
$\{|M_1|,|M_2|,\ldots, |M_t|\}$, where $M_1\subset M_2\subset \ldots
\subset M_t$ is a chain of indecomposable submodules of $M$. We call
$\mu(M)$ the {\bf Gabriel-Roiter measure}  of $M$. If $M$ is an
indecomposable $\Lambda$-module, we call an inclusion $T\subset M$
with $T$ indecomposable a {\bf GR inclusion} provided
$\mu(M)=\mu(T)\cup\{|M|\}$, thus if and only if every proper
submodule of $M$ has Gabriel-Roiter measure at most $\mu(T)$. In
this case, we call $T$ a {\bf GR submodule} of $M$.
A subset $I$ of $\mathbb{N}$ is called a GR measure of $\Lambda$ if there is
an indecomposable module $M$ with $\mu(M)=I$.

\begin{remark} Although in the introduction we define
the Gabriel-Roiter measure in a different way, these two  definitions (orders) coincide.
In fact, for each
$I=\{a_i|i\}\in\mathcal{P}(\mathbb{N})$, let
$\mu(I)=\sum_{i}\frac{1}{2^{a_i}}$. Then $I<J$ if and only if
$\mu(I)<\mu(J)$.
\end{remark}

We denote by $\mathcal{T}$, $\mathcal{C}$ and $\mathcal{L}$ the
collection of indecomposable $\Lambda$-modules, which are in the take-off part,
the central part and the landing part, respectively. Now we present one result
concerning Gabriel-Roiter measures, which will be used later on.
For more basic properties we refer to \cite{R1,R2}.

\begin{prop}\label{epi} Let $\Lambda$ an Artin algebra and $X\subset M$ be a GR
inclusion.
\begin{itemize}
  \item[(1)] If $\mu(X)<\mu(Y)<\mu(M)$,
                          then $|Y|>|M|$.
  \item[(2)] There is an irreducible monomorphism $X\ra Y$ with $Y$
            indecomposable and an epimorphism $Y\ra M$.
\end{itemize}
\end{prop}
The first statement is a direct consequence of the definition. For a
proof of the second statement, we refer to \cite[Proposition
3.2]{Ch2}.

\subsection{}Let $Q$ be a tame quiver of type
$\widetilde{\mathbb{A}}_n,n\geq 1$, $\widetilde{\mathbb{D}}_n,n\geq 4$, $\widetilde{\mathbb{E}}_6$,
$\widetilde{\mathbb{E}}_7$ or $\widetilde{\mathbb{E}}_8$, and $\Rep(Q)$ the category of
finite dimensional representations over an algebraically closed
field. We simply call the representations in $\Rep(Q)$ modules. We
briefly recall some notations and refer to \cite{ARS, DR} for
details. If $X$ is a quasi-simple module, then there is a unique
sequence $X=X_1 \ra X_2\ra \ldots\ra X_r\ra\ldots$ of irreducible
monomorphisms. Any indecomposable regular module $M$ is of the
form $M\cong X_i$ with $X$  a quasi-simple module (quasi-socle of
$M$) and $i$ a natural number (quasi-length of $M$). The  rank of an
indecomposable regular module $M$ is the minimal positive integer
$r$ such that $\tau^rM=M$, where $\tau=DTr$ denotes the Auslander-Reiten translation.
A regular component (standard stable tube) of the
Auslander-Reiten quiver of $Q$ is called exceptional if its rank
(the rank of any quasi-simple module on it) $r>1$.  If $X$ is
quasi-simple of rank $r$, then the dimension vector $\udim
X_r=\delta=\sum_{i=1}^r\tau^rX$, where $\delta$ is the minimal
positive imaginary root.  Let $|\delta|=\sum_{\nu\in Q}
\delta_\nu$. In particular, if $Q$ is of type $\widetilde{\mathbb{A}}_n$, then
$\delta_\nu=1$ for each $\nu\in Q$ and $|\delta|=n+1$.  A quasi-simple module of rank 1 will be called a
homogeneous quasi-simple module. We denote by $H_i$ an indecomposable
homogeneous module with quasi-length $i$. (There are infinitely many
homogeneous tubes.  However, the GR measure $\mu(H_i)$ does not depend on the
choice of $H_i$.) We denote by
$\mathcal{P}$, $\mathcal{R}$ and $\mathcal{I}$ the collection of
indecomposable preprojective, regular and preinjective modules,
respectively.

We collect some known facts in the following proposition, which will
be quite often used in our later discussion. The proofs can be found
in \cite[Section 3]{Ch3}.
\begin{prop}\label{bigprop} Let $Q$ be a tame quiver of type $\widetilde{\mathbb{A}}_n$, $n\geq 1$.
\begin{itemize}
    \item[(1)]  Let $\iota:T\subset M$  be a GR inclusion.
               \begin{itemize}
                   \item[a] If $M\in\mathcal{P}$, then $\iota$ is an irreducible monomorphism.
                   \item[b] If $M\in\mathcal{R}$ is a quasi-simple module,
                             then $T\in\mathcal{P}$.
                   \item[c] If $M=X_i\in\mathcal{R}$ with $X$
                              quasi-simple and $i>1$, then $T\in\mathcal{P}$ or $T\cong X_{i-1}$.
                   \item[d] If $M\in\mathcal{I}$, then
                              $T\in\mathcal{R}$.
               \end{itemize}
    \item[(2)] If $X\in\mathcal{P}$, then $X\in\mathcal{T}$ and $\mu(X)<\mu(H_1)$.
    \item[(3)] Let $H_1$ be a homogeneous quasi-simple module. Then $\mu(H_1)$ is a central measure and $\mu(M)>\mu(H_1)$
               if $M\in\mathcal{I}$ satisfies $\udim M>\delta$.
    \item[(4)] Let $X$ be a quasi-simple module of rank $r$. Then
               \begin{itemize}
                   \item[a] If $\mu(X_r)< \mu(H_1)$, then
                              $\mu(X_i)<\mu(H_j)$ for all $i\geq 1$ and $j\geq 1$.
                   \item[b] If $\mu(X_r)\geq \mu(H_1)$, then $X_{i}$
                              is the unique  GR
                              submodule of $X_{i+1}$ for every $i\geq r$. If in addition $r>1$, then
                              $\mu(X_i)>\mu(H_j)$ for all $i>r$ and $j\geq
                              1$.
                \end{itemize}
    \item[(5)] Let $\mathbb{T}$ be a stable tube of rank $r>1$. Then there is
              a quasi-simple module $X$ on $\mathbb{T}$ such that
              $\mu(X_r)\geq \mu(H_1)$.
    \item[(6)] Let $S$ be a quasi-simple module of rank $r$ which is also a
              simple module. Then $\mu(S_r)<\mu(H_1)$, and thus
                 $\mu(S_j)<\mu(H_1)$ for all $j\geq 1$.
    \item[(7)] Let $M\in \mathcal{I}\setminus\mathcal{T}$ and $N$  be
a  GR submodule of $M$. Thus  $N\cong X_i$ for some quasi-simple module $X$ by (1)d.  Then $\mu(M)>
\mu(X_j)$ for all $j\geq 1$.
\end{itemize}
\end{prop}

\begin{remark}
Some statements in Proposition \ref{bigprop} hold in general.
For example, the statements (2), (4) and (7) and the first argument of the statement
(3) also hold \cite{Ch4} for tame quivers of type
$\widetilde{\mathbb{D}}_n$, $\widetilde{\mathbb{E}}_6$,
$\widetilde{\mathbb{E}}_7$ and $\widetilde{\mathbb{E}}_8$.
The statement (5) is known to be true \cite{Ch5} for tame quivers which is not of the $\widetilde{\mathbb{E}}_8$ type.
\end{remark}

\begin{lemm}\label{onemap}
 Let $Q$ be a tame quiver of type
$\widetilde{\mathbb{A}}_n$.  Then for every indecomposable
preinjective module $M$, there is, in each regular component, precise
one quasi-simple module $X$ such that $\Hom(X,M)\neq 0$. In
particular, up to isomorphism, each indecomposable preinjective
module contains in each regular component at most one GR submodule.
\end{lemm}
\begin{proof} Let $M=\tau^sI_\nu$, where $I_\nu$ is an indecomposable
injective module corresponding to a vertex $\nu\in Q$. It is obvious
that there is a quasi-simple module $X$ on a given regular component such
that $\Hom(X,I_\nu)\neq 0$. Thus $\Hom(\tau^sX,M)\neq 0$. Assume
that $X$ and $Y$ are non-isomorphic quasi-simple modules on the same
tube such that $\Hom(X,M)\neq 0\neq \Hom(Y,M)$. Then
$\Hom(\tau^{-s}X,I_\nu)\neq 0\neq \Hom(\tau^{-s}Y,I_\nu)$. Thus
$(\udim\tau^{-s}X)_\nu\neq 0\neq (\udim\tau^{-s}Y)_\nu$, which is
impossible since $1=\delta_\nu\geq
(\udim\tau^{-s}X)_\nu+(\udim\tau^{-s}Y)_\nu$.
\end{proof}

\section{the number of GR submodules}\label{string}
As we have mentioned in the introduction, the number of the GR submodules of
a given indecomposable module over a representations-finite algebra
is conjectured to be bounded by three.  In this section, we will show
for a string algebra that this number is always bounded by two for any
indecomposable string module containing no band submodules.  We will also describe the numbers of the GR
submodules for different kinds of indecomposable modules over
tame quivers $Q$ of type $\widetilde{\mathbb{A}}_n$.

\subsection{String modules} We first recall what string modules are.
For details, we refer to \cite{BR}. Let $\Gamma$ be a string
algebra with underlying quiver $Q$. We denote by $s(C)$ and $e(C)$ the starting and the ending
vertices of a given string $C$, respectively.  Let
$C=c_nc_{n-1}\cdots c_2c_1$ be a string, the corresponding string
module $M(C)$ is defined as follows: let $u_i=s(c_{i+1})$ for $0\leq
i\leq n-1$ and $u_n=e(c_n)$. For a vertex $\nu\in Q$, let
$I_\nu=\{i|u_i=\nu\}\subset\{0,1,\ldots,n\}$. Then the vector space $M(C)_\nu$
associated to $\nu$  satisfies that $\dim M(C)_\nu=|I_\nu|$
and has $z_i,i\in I_\nu$ as basis.  If for  $1\leq i\leq n$, the symbol $c_i$ is an arrow $\beta$,
define $\beta(z_{i-1})=z_i$. If for
$1\leq i\leq n$, the symbol $c_i$ is an
inverse of an arrow $\beta$, define $\beta(z_i)=z_{i-1}$.
Note that the indecomposable string modules are
uniquely determined by the underlying string,  up to the equivalence
given by $C\sim C^{-1}$.

If $C=E\beta F$ is a string with $\beta$ an arrow, then the string
module $M(E)$ is a submodule of $M(C)$: let $E$ be of length $n$ and
$F$ be of length $m$. then $C$ has length $n+m+1$. If $M(C)$ is
given by  $n+m+2$ vectors $z_0,z_1,\ldots,z_{n+m+2}$, it is obvious
that the space determined by the vectors $z_0,z_1,\ldots, z_n$ defines a
submodule, which is $M(E)$. The corresponding factor module is
$M(F)$. If $C=E\beta^{-1} F$ is a string with $\beta$ an arrow, we
may obtain similarly an indecomposable submodule $M(F)$ of $M(C)$ with factor
module $M(E)$.

\subsection{A``covering" of a string module} Let
$C=c_nc_{n-1}\cdots c_2c_1$ be a string. We associate
with $C$ a Dynkin quiver $\mathbb{A}_{n+1}$ as follows: the vertices
are $u_i$, and there is an arrow $u_{i-1}\stackrel{\alpha_i}{\ra}
u_i$ if $c_i$ is an arrow, and an arrow
$u_{i}\stackrel{\alpha_i}{\ra} u_{i-1}$ in case $c_i$ is an inverse
of an arrow.  Let $M(C)$ be the string module and
$M_{\mathbb{A}}(C)$ be the unique sincere indecomposable
representation over $\mathbb{A}_{n+1}$.

Before going further, we introduce the follow lemma, which was proved in \cite{R2}.
\begin{lemm}\label{sing}
Let $M$ and $N$ be indecomposable modules over an Artin algebra $\Lambda$
and $\sing(N,M)$ be the set of all non-injective homomorphisms. If $N$ is a GR submodule of $M$,
then  $\sing(N,M)$ is a subspace of $\Hom_\Lambda(N,M)$.
\end{lemm}

From now on, we assume that $M(C)$ is a string module over $\Gamma$
determined by a string $C$ such that $M(C)$ contains no band submodules.
Thus any submodule of $M$ is a string module.

\begin{lemm}\label{basis} Let $N$ be a GR submodule of $M(C)$.
Then there is a substring $C'$ of $C$, such that
the submodule $X$ of $M(C)$ determined by $C'$ is isomorphic to $N$.
\end{lemm}
\begin{proof}
Let $f: N\subset M(C)$ be the inclusion. Then $f$ is a linear combination of a basis described
in \cite{CB}, say $f_1,\ldots, f_t$. Since by Lemma \ref{sing} $\sing(N,M)$ is a subspace ,  there is an $1\leq i\leq t$ such
that $f_i$ is a monomorphism.  By the description of the basis,
$X=\im f_i$ is a submodule of $M$ determined
by a substring $C'$ of $C$. It is obvious that $N\cong X$.
\end{proof}

By this lemma, we may assume without loss of generality that a GR submodule $N$ of $M(C)$ is always
given by a substring of $C$.
Thus we may obtain in an obvious way an indecomposable  submodule $\mathcal{G}(N)$
of $M_\mathbb{A}(C)$ using the construction of the quiver $\mathbb{A}_{n+1}$.

\begin{remark}
Note that different monomorphisms $f_i$ in the basis in the proof of Lemma \ref{basis} may
give rise to different $\mathcal{G}(N)$, which are indecomposable and pairwise non-isomorphic.
They all have
the same length $|N|$ by the construction.
\end{remark}

More generally, let $N_1\subset N_2\subset \ldots N_s=N\subset N_{s+1}=M(C)$ be a GR filtration of $M(C)$.
By above discussion, we may assume that there is a sequence of substrings
$C_1\subset C_2\subset \ldots \subset C_s\subset C$ such that $N_i$
is determined by the substring $C_i$.
Thus we may define
$\mathcal{G}(N_i)$  using the construction of the quiver $\mathbb{A}_{n+1}$.
Therefore, we get inclusions of indecomposable modules over the quiver
$\mathbb{A}_{n+1}$:  $\mathcal{G}(N_1)\subset \mathcal{G}(N_2)\subset \ldots\subset \mathcal{G}(N_s)\subset M_\mathbb{A}(C)$.
Conversely,
if $T$ is a submodule of $M_\mathbb{A}(C)$ with inclusion map $f$, then there is a natural way to get a string submodule
$\mathcal{F}(T)$ of $M(C)$. We may also denote the inclusion by $f:\mathcal{F}(T)\subset M$. Then
under the inclusion, $\mathcal{G}\mathcal{F}(T)=T$. It is easily seen that
$\mathcal{F}(\mathcal{G}(N)\cong N$.
Note that $\mathcal{F}$ preserves indecomposables, monomorphisms and lengths. In particular,
$\mu(T)\leq \mu(F(T))$.

\begin{lemm}\label{propF}  We keep the notations as above.
  \begin{itemize}
    \item[(1)] For each $i$, $\mathcal{G}(N_i)$ is a GR submodule of $\mathcal{G}(N_{i+1})$ and $\mu(\mathcal{G}(N_i))=\mu(N_i)$.
    \item[(2)]  If $T$ is a GR submodule of $M_\mathbb{A}(C)$, then $\mathcal{F}(T)$ is a GR submodule of
                $M$.
  \end{itemize}
\end{lemm}
\begin{proof}
(1) We use induction.
Assume first that $i=1$ and that $\mathcal{G}(N_1)$ is not a GR submodule of $\mathcal{G}(N_2)$.
Let $X\subset \mathcal{G}(N_2)$ be a GR submodule and thus $\mathcal{F}(X)$ is isomorphic to a submodule of $N_2$. Since $\mathcal{F}$ preserves monomorphisms,
$\mu(N_1)\geq \mu(\mathcal{F}(X))\geq \mu(X)>\mu(\mathcal{G}(N_1)$. This is impossible since $N_1$ and $\mathcal{G}(N_1)$ are both simple modules.   Thus $\mathcal{G}(N_1)$ is a GR submodule of $\mathcal{G}(N_{2})$.
It follows $\mu(\mathcal{G}(N_2))=\mu(N_2)$ since $|\mathcal{G}(N_2)|=|N_{2}|$.

Now assume that $\mu(\mathcal{G}(N_i))=\mu(N_i)$.  Let $X$ be a GR submodule of $\mathcal{G}(N_{i+1})$.
Then $\mu(\mathcal{G}(N_i))\leq\mu(X)\leq \mu(F(X))\leq \mu(N_i)$. Therefore, $\mu(\mathcal{G}(N_i))=\mu(X)$
by induction and $\mathcal{G}(N_i)$ is a GR submodule of $\mathcal{G}(N_{i+1})$. Hence
$\mu(\mathcal{G}(N_{i+1}))=\mu(N_{i+1})$ since $|\mathcal{G}(N_{i+1})|=|N_{i+1}|$.

(2) Since  $N$ is a GR submodule of $M(C)$, $\mu(T)=\mu(\mathcal{G}(N))=\mu(N)\geq\mu(\mathcal{F}(T))\geq \mu(T)$.
Thus $\mu(N)=\mu(\mathcal{F}(T))$ and $\mathcal{F}(T)$ is a GR submodule of $M(C)$.
\end{proof}

For a Dynkin quiver of type $\mathbb{A}$, we have shown in
\cite{Ch1} the following result:
\begin{prop} Let $Q$ be a Dynkin quiver of type $\mathbb{A}$. Then each
indecomposable module  has at most two GR submodules and each factor
of a GR inclusion is a uniserial module.
\end{prop}

As a consequence of this proposition and Lemma \ref{propF},  we have
\begin{theo} Let $\Lambda$ be a string algebra and  $M(C)$ be a string module
containing no band submodules. Then $M(C)$ contains, up to
isomorphism, at most two GR submodules and the factors of the GR
inclusions are uniserial modules.
\end{theo}

\begin{coro} If $\Lambda$ is a representation-finite string algebra,
then each indecomposable module has, up to isomorphism, at most two GR
submodules and the GR factors are uniserial.
\end{coro}

\subsection{} Now we assume that $Q$ is a tame quiver of type
$\widetilde{\mathbb{A}}_n$. Then every indecomposable regular module
with rank $r>1$ is a string module containing no band submodules, thus
has at most two GR submodules up to isomorphism by above theorem.

\begin{prop}\label{2gr} If an exceptional regular module has precisely two GR
submodules, then one of the GR inclusions is an irreducible map. In particular,
every exceptional regular module has at most one preprojective GR submodule, up to isomorphism.
\end{prop}

Before proving this proposition, we briefly recall
how the irreducible monomorphisms between string modules
look like. We refer to \cite{BR} for details.
Let $C=c_r\cdots c_2c_1$ be a string such that $c_r$ is an arrow.
If there is an arrow $\gamma$ with $e(\gamma)=e(c_r)$, let $D=d_t\cdots d_2d_1$
be a string with $s(D)=s(d_1)=s(\gamma)$ such that $d_i$ is an arrow for every $i$ and such that
$t\geq 0$ is maximal.  Then the natural inclusion $M(C)\ra M(D\gamma^{-1} C)$ is an irreducible monomorphism.
Similarly, Let $C=c_r\cdots c_2c_1$ be a string such that $c_1$ is an arrow.
If there is an arrow $\gamma$ with $e(\gamma)=s(c_1)$,  let $D=d_t\cdots d_2d_1$
be a string with $e(D)=e(d_t)=s(\gamma)$ such that $d_i$ is an inverse of an arrow for every $i$ and such that $t\geq 0$ is maximal. Then  the natural inclusion $M(C)\ra M(C\gamma D)$ is an irreducible monomorphism.
Any irreducible monomorphism between string modules is of one of these forms.

\begin{proof} Let $M(C)$ be an exceptional regular module with
$C=c_m\cdots c_2c_1$,  which has precisely (up to isomorphism) two
GR submodules. Then the module $M_\mathbb{A}(C)$ also has two GR
submodules, which are actually given by the irreducible monomorphism
$X\ra M_\mathbb{A}(C)$ and $Y\ra M_\mathbb{A}(C)$.  By definition of
$M_\mathbb{A}(C)$, we may identify the arrows $\alpha_i$ or its
inverse in $\mathbb{A}_{m+1}$with $c_i$ in the string $C$ of
$\widetilde{\mathbb{A}}_n$. We may assume that $X$  is determine by
string $E$ and $M_\mathbb{A}(C)$ is determined by $F\alpha^{-1}E$,
where $F$ is a composition of arrows or a trivial path and $\alpha$
is an arrow. Then under the above identification, we have
$C=F\alpha^{-1}E$. Let $M(C)\ra M'$ be the unique irreducible
monomorphism with $M'$ determined by a string
$F'\beta^{-1}F\alpha^{-1}E$, where $ F'$ is a compositions of arrows
or a trivial path and $\beta$ is an arrow. Thus either the ending
vertex $e(F)$ is  a sink, or $F$ is a trivial path. Again by the
description of irreducible monomorphism, the inclusion
$\mathcal{F}(X)\ra M(C)$ is still an irreducible map. Note that
$\mathcal{F}(X)$ is a GR submodule of $M(C)$ by Lemma \ref{propF}.
\end{proof}

\begin{remark} Let $Q$ be a tame quiver of type
$\widetilde{\mathbb{A}}_n$ and $M$ be a non-simple indecomposable
module. Let  $gr(M)$ denote  the number of the isomorphism classes of the GR submodules of $M$.

\begin{itemize}
   \item [(1)] If $M$ is preprojective, each GR inclusion $X\subset M$
               is namely an irreducible map (Proposition \ref{bigprop}(1)a).
               In particular, $gr(M)\leq 2$ since
               there are precisely two irreducible maps to $M$, which are monomorphisms.
   \item[(2)]  If $M$ is a quasi-simple module of rank $r>1$, then
                $gr(M)=1$ since $M$ is uniserial.
   \item[(3)]  If $M$ is a non-quasi-simple regular module of rank $r>1$, then $gr(M)\leq 2$, and
              one of the GR inclusion is irreducible in case $gr(M)=2$ (Proposition \ref{2gr}).
   \item[(4)]  If $M=X_i$ is a regular module with $i>1$, where $X$ is a
              quasi-simple module of rank $r\geq 1$ with $\mu(X_r)\geq \mu(H_1)$,
               then $gr(M)=1$ and the unique
              GR inclusion is an irreducible map (Proposition \ref{bigprop}(4)b).
   \item[(5)]   If $M$ is preinjective, then $M$ contains, up to isomorphism,
                at most one GR submodule in each regular component (Lemma \ref{onemap}). If
                we identify the homogeneous modules $H_i$, then $gr(M)\leq 3$. Under the convention,
              if $gr(M)=3$, then $M$ contains a homogeneous module $H_i$ and as well as an exceptional
               module $X_j$ with rank $r>1$ as GR submodules. Note that $\mu(X_{mr})\neq \mu(H_s)$ for any $m>1$ and any $s>1$.
               It follows that $i=1$ and $j=r$ and thus $\delta=\udim X_j=\udim H_i$.  Therefore, $|M|<2|\delta|$ (Proposition \ref{epi}(2)).
   \item[(6)] A homogeneous simple module $H_1$ may contains more GR submodules. For example,
             if $n$ is odd and $Q$ is with a
             sink-source orientation (see \cite[Example 3]{Ch3}),
             then the GR measure of a homogeneous simple module is
             $\mu(H_1)=\{1,3,5,\ldots, n,n+1\}$.
             There are up to isomorphism $\frac{n+1}{2}$ indecomposable preprojective modules with length $n$ and
             they are all non-isomorphic GR submodules of $H_1$.
             In general, $gr(H_1)$ is bounded by the number of the indecomposable summands
             of the projective cover of $H_1$.
   \item[(7)] One may also define $gr(M)$ as the number of the dimension vectors of the GR submodules of $M$.
               Then it is easily seen that $gr(M)\leq 2$ for each indecomposable module $M$, which is not a homogeneous quasi-simple module $H_1$.

\end{itemize}

\end{remark}

\section{Direct predecessor}\label{predecessors}

Let $\Lambda$ be an Artin algebra and $I$ and $J$ be two different GR measures for $\Lambda$. 
Then $J$ is called a direct successor of $I$ if, first, $I<J$ and second, there does not
exist a GR measure $J'$ with $I<J'<J$. It is easily seen that if $J$
is the direct successor of $I$, then $J$ is a take-off (resp.
central, landing) measure if and only if so is $I$. Let $I^1$ be the
largest GR measure, i.e. the GR measure of an indecomposable
injective module with maximal length. It was proved in  \cite{R2}
that any Gabriel-Roiter measure $I$ different from $I^1$ has a
direct successor. However, there are GR measures, which does not
admit a direct predecessor. By the construction of the take-off
measures and the landing measures \cite{R1}, the GR measures
having no direct predecessors are central measures.

\subsection{} From now on, we fix a tame quiver $Q$  of type
$\widetilde{\mathbb{A}}_n$. The
following proposition gives a GR measure possessing no direct
predecessor.

\begin{prop}\label{H1}The GR measure $\mu(H_1)$ of a homogeneous quasi-simple module $H_1$  has no
direct predecessor.
\end{prop}

\begin{proof}
Assume, to the contrary, that $\mu(M)$ is the
direct predecessor of $\mu(H_1)$ for some indecomposable module $M$.
Since $\mu(H_1)$ is a central measure, so is $\mu(M)$.  It follows
that $M$ is not preprojective. Let $Y$ be a GR submodule of $H_1$.
Since $Y$  is preprojective,  $\mu(Y)<\mu(M)<\mu(H_1)$ and thus
$|M|>|H_1|$ (Proposition \ref{epi}(1)). If $M$ is preinjective, then there is a monomorphism
$H_1\ra M$ because $|M|>|H_1|$, and hence $\mu(H_1)<\mu(M)$. This
contradiction implies that  $M$ is a regular module. Assume that
$M=X_i$ for some quasi-simple module $X$ of rank $r>1$. Because
$|X_i|=|M|>|H_1|$, we have $i> r$. Therefore, $\mu(X_r)<\mu(M)<\mu(H_1)$.
It follows from that $\mu(M)<\mu(X_j)<\mu(H_1)$ for all $j>i$ (Proposition \ref{bigprop}(4)a). This is a
contradiction.
\end{proof}

\begin{prop}\label{prepre} Let $M\in\mathcal{I}\setminus\mathcal{T}$. If
$\mu(N)$ is the direct predecessor of $\mu(M)$ for some
indecomposable module $N$, then $N\in\mathcal{I}$ and $|N|>|M|$.
\end{prop}

\begin{proof}Since $\mu(N)$ is not a take-off measure, $N$ is not preprojective. Assume
for a contradiction that $N=Y_j$ is regular for some quasi-simple
module $Y$. Let $X_i$ be a GR submodule of  $M$ for some
quasi-simple module $X$ and some $i\geq 1$ (Proposition \ref{bigprop}(1)d). Then $\mu(M)>\mu(X_t)$
for all $t\geq 1$  by
Proposition \ref{bigprop}(7) and thus $\mu(X_t)\neq \mu(Y_j)$ for any $t\geq 1$. Therefore $\mu(X_i)<\mu(Y_j)<\mu(M)$. It follows
that $|Y_j|>|M|$ and $\mu(M)<\mu(Y_{j+1})$ since $\mu(N)=\mu(Y_j)$
is the direct predecessor of $\mu(M)$.  Notice that a GR submodule $T$
of $Y_{j+1}$ is either $Y_j$ or a preprojective module. In
particular $\mu(T)<\mu(M)<\mu(Y_{j+1})$. Thus $|M|>|Y_{j+1}|$ by Proposition \ref{epi}(1). This
contradicts $|N|=|Y_j|>|M|$. Therefore, $N$ is preinjective.
\end{proof}

\subsection{}A regular module $X_i$, $i>r$, for some
quasi-simple module $X$ of rank $r$  may contain a
preprojective module as a GR submodule. However, this cannot happen
if $\mu(X_r)\geq \mu(H_1)$, in which case the irreducible monomorphisms
$X_r\ra X_{r+1}\ra X_{r+2}\ra\ldots$ are GR inclusions (Proposition \ref{bigprop}(4)b).
The GR measure $\mu(X_r)$
for a quasi-simple module $X$ of rank $r$ is important when
comparing the GR measures of regular modules $X_i$ and those of
homogeneous modules $H_j$. Namely, there is a similar result that
can be used to compare the GR measures of two non-homogeneous
regular modules.

\begin{lemm}\label{2} Let $X,Y$ be quasi-simple modules of rank $r$ and $s$,
respectively. Assume that $\mu(X_r)\geq \mu(H_1)$.
\begin{itemize}
       \item[(1)]  If $\mu(X_r)>\mu(Y_s)$, then $\mu(X_i)>\mu(Y_j)$
                 for all $i\geq r$, $,j\geq 1$.
      \item[(2)]  If $\mu(X_i)=\mu(Y_j)$ for some $i\geq 2r$, then $r=s$
                 and $\mu(X_t)=\mu(Y_t)$ for every $t\geq r$.
       \item[(3)] If $\mu(X_{2r})>\mu(Y_{2s})$, then
                 $\mu(X_i)>\mu(Y_j)$ for all $i\geq 2r, j\geq 1$.
\end{itemize}
\end{lemm}
\begin{proof}(1) If $\mu(Y_s)<\mu(H_1)$, then $\mu(Y_j)<\mu(H_1)$ for
all $j\geq 1$.  Thus we may assume that $\mu(Y_s)\geq\mu(H_1)$.
Since for each $j\geq s$, $\mu(Y_j)$ starts with $\mu(Y_s)$  and
$|Y_s|=|X_r|=|\delta|$, we have $\mu(X_r)>\mu(Y_j)$.

(2)
By assumption, we have $j\geq 2s$ because $|Y_j|=|X_i|\geq 2|\delta|$.
Since $\mu(X_r)\geq \mu(H_1)$, we have $\mu(Y_s)\geq \mu(H_1)$ and the irreducible
monomorphisms in the sequences
$$X_r\ra X_{r+1}\ra X_{r+2}\ra\ldots,\,\ \textrm{and}\,\  Y_s\ra Y_{s+1}\ra Y_{s+2}\ra\ldots$$
are all
GR inclusions (Proposition \ref{bigprop}(4)).
It follows that
$$\begin{array}{rclll}
\mu(Y_{j})&=&\mu(Y_s)\cup\{|Y_{s+1}|,|Y_{s+2}|,\ldots,|Y_{2s-1}|,|Y_{2s}|,|Y_{2s+1}|,\ldots,|Y_{j}|\}&&\\
&=&\mu(X_r)\cup\{|X_{r+1}|,|X_{r+2}|,\ldots,|X_{2r-1}|,|X_{2r}|,|X_{2r+1}|,\ldots,|X_{i}|\}&=&\mu(X_i).
\end{array}$$
Since $|X_r|=|Y_s|=|\delta|$ and $|X_{2r}|=|Y_{2s}|=2|\delta|$,
so $\mu(X_r)=\mu(Y_s)$ and
$\mu(X_{2r})=\mu(Y_{2s})$. Note that
$$|X_{r+l}|-|X_{r+l-1}|=|Y_{r+l}|-|Y_{r+l-1}|$$ for all $l\geq 1$.
It follows  that $r=s$ and $\mu(X_t)=\mu(Y_t)$ for all $t\geq r=s$.

(3) follows similarly.
\end{proof}

\begin{coro}Let $X$ be a quasi-simple module of rank $r$ such that
$\mu(X_r)\geq \mu(H_1)$. If $M$ is an indecomposable module such
that $\mu(M)=\mu(X_i)$ for some $i\geq 2r$, then $M$ is a regular
module.
\end{coro}

\begin{proof}
Assume, to the contrary, that $M$ is preinjective. Let $Y_t$ be a
GR submodule of $M$ for some quasi-simple module $Y$ of rank $s$.
Then $\mu(M)>\mu(Y_j)$ for all $j\geq 1$ by Proposition
\ref{bigprop}(7).  Thus $Y\ncong X$ and $t\geq 2s$ since
$|M|=|X_i|>2|\delta|$ and since there is an epimorphism $Y_{t+1}\ra M$ by Proposition \ref{epi}(2).
Notice that  $\mu(Y_s)\geq \mu(H_1)$. Otherwise,
we would have $\mu(Y_{t})<\mu(H_1)$. However, there is a monomorphism $H_1\ra M$ since
$|M|>2|\delta|$. We obtain a contradiction because $Y_t$ is a GR submodule of $M$.
Since $Y_t$ is a GR submodule of $N$ and $\mu(M)=\mu(X_i)$, so $\mu(Y_t)=\mu(X_{i-1})$.
Therefore, $r=s$ and
$\mu(Y_{t+1})=\mu(X_i)$ by above lemma. This contradicts
$|Y_{t+1}|>|M|=|X_i|$ (Proposition \ref{epi}(1)). Thus $M$ is regular.
\end{proof}

\subsection{}We have seen in Proposition \ref{bigprop}(4) that
the irreducible maps $H_1\ra H_2\ra H_3\ra\ldots$ are GR inclusions.
One can show more: namely, in \cite{Ch3} ( or \cite{Ch4} for
general cases) we proved that $\mu(H_{i+1})$ is the direct successor
of $\mu(H_i)$ for each $i\geq 1$. Let $X$ be a quasi-simple module of rank
$r>1$. It is possible (for example, if $\mu(X_r)\geq \mu(H_1)$) that
all irreducible maps $X_r\ra X_{r+1}\ra X_{r+2}\ra\ldots$ are GR
inclusions. However, it is not true in general that $\mu(X_{j+1})$
is the direct successor of $\mu(X_j)$ for all $j\geq r$ (\cite[Example 4]{Ch3}).
The following proposition tells if $\mu(X_r)\geq
\mu(H_1)$  and if $\mu(X_{j+1})$ is not the direct successor of
$\mu(X_j)$, then $j< 2r$.

\begin{prop}\label{ds} Let $X$ be a quasi-simple module of rank $r$ such that
$\mu(X_r)\geq \mu(H_1)$. Then $\mu(X_{j+1})$ is a direct successor
of $\mu(X_{j})$ for each $j\geq 2r$.
\end{prop}

\begin{proof}
We may assume  $r>1$.  We first show that there does not exist an
indecomposable regular module $M$ such that $\mu(M)$ lies between
$\mu(X_j)$ and $\mu(X_{j+1})$ for any $j\geq 2r$. Assume, to the contrary, that there exists a $j\geq 2r$ and an
indecomposable regular module $M$ with
$\mu(X_j)<\mu(M)<\mu(X_{j+1})$.  We may assume that $|M|$ is
minimal. Then $|M|>|X_{j+1}|>2|\delta|$, since $X_j$ is a GR
submodule of $X_{j+1}$.   Let $M=Y_i$ for some quasi-simple module
$Y$ of rank $s>1$. It follows that $\mu(Y_s)\geq \mu(H_1)$ and $i>
2s$. Therefore, $Y_{i-1}$ is a GR submodule of $Y_i$ and
$$\mu(Y_{i-1})\leq \mu(X_j)<\mu(M)=\mu(Y_i)<\mu(X_{j+1})$$
by minimality of $|M|$. This implies $\mu(Y_{i-1})=\mu(X_j)$, since
otherwise $|X_j|>|M|>|X_{j+1}|$, which is a contradiction. Observe
that $i-1\geq 2s$ and $j\geq 2r$.  Then Lemma \ref{2} implies
$\mu(X_t)=\mu(Y_t)$ for all $t\geq r=s$. This contradicts the
assumption $\mu(X_j)<\mu(M)=\mu(Y_i)<\mu(X_{j+1})$. Therefore, there
are no indecomposable regular modules $M$ satisfying
$\mu(X_j)<\mu(M)<\mu(X_{j+1})$ for any $j\geq 2r$.

Now we assume that $M$ is an indecomposable preinjective module such
that $\mu(X_j)<\mu(M)<\mu(X_{j+1})$. Since $\mu(X_r)\geq \mu(H_1)$, so
$X_j$ is a GR submodule of $X_{j+1}$. It follows that $|X_{j+1}|>|M|>|X_j|$
by Proposition \ref{epi}(1).  Let $Y_i$ be a GR submodule of
$M$ for some quasi-simple module $Y$ and $i\geq 1$. Then $Y_i\ncong
X_t$ for any $t>0$ by Proposition \ref{bigprop}(7). Comparing the
lengths, we have $\mu(Y_i)\geq \mu(X_j)$. (Namely, if $\mu(Y_i)<\mu(X_j)<\mu(M)$,
then $|X_j|>|M|$ by Proposition \ref{epi}(1) since $Y_i$ is a GR submodule of $M$.
But on the other hand, $|X_j|<|M|$ by previous discussion.)
 Thus Proposition
\ref{bigprop}(7) implies that
$\mu(X_j)<\mu(Y_{i+1})<\mu(M)<\mu(X_{j+1})$. Therefore, we get an
indecomposable regular module $Y_{i+1}$ with GR measure lying
between $\mu(X_j)$ and $\mu(X_{j+1})$, which is a contradiction.

The proof is completed.
\end{proof}

\subsection{}
Let $X$ be a quasi-simple module of rank $r$ such that $\mu(X_r)\geq
\mu(H_1)$. For a given $i\geq 2r$, let
$\mu_{i,1}>\mu_{i,2}>\ldots>\mu_{i,t_i}$ be all  different GR
measures of the form $\mu_{i,j}=\mu(X_i)\cup\{a_{i,j}\}$ and
$a_{i,j}\neq |X_{i+1}|$ for any $1\leq j\leq t_i$. Notice that there
are only finitely many such $\mu_{i,j}$ for each given $i$.

\begin{lemm}\label{befdp}
\begin{itemize}
   \item[(1)] $a_{i,j}<|X_{i+1}|$ for all $1\leq j\leq t_i$ and $a_{i,j}<a_{i,l}$ if $1\leq j<l<t_i$.
   \item[(2)] $\mu_{i,j}>\mu(X_t)$ for all $1\leq j\leq
              t_i, t\geq 1$.
   \item[(3)] $\mu_{i,j}>\mu_{l,h}$ if $i< l$.
   \item[(4)] If $M$ is an indecomposable module such that $\mu(M)=\mu_{i,j}$, then $M\in\mathcal{I}$.
\end{itemize}
\end{lemm}
\begin{proof}
(1) If $a_{i,j}>|X_{i+1}|$, then
  $$\mu_{i,j}=\mu(X_{i})\cup\{a_{i,j}\}<\mu(X_i)\cup\{|X_{i+1}|\}=\mu(X_{i+1}).$$
This contradicts $\mu(X_{i+1})$ is a direct successor of $\mu(X_i)$
(Proposition \ref{ds}). Thus $a_{i,j}<|X_{i+1}|$.

(2) follows from (1) and the fact that $X_{2r}\subset
X_{2r+1}\subset\ldots\subset X_t\subset X_{t+1} \subset\ldots$ is a
sequence of GR inclusions.

(3) If $i<l$, then
$$\begin{array}{rcl}
\mu_{l,h} &=& \mu(X_{l})\cup\{a_{l,h}\}\\
          &=& \mu(X_{i})\cup\{|X_{i+1}|,\ldots,|X_l|,a_{l,h}\}\\
          &<& \mu(X_{i})\cup\{a_{i,j}\}\\
          &=& \mu_{i,j}
\end{array}$$

(4) If $M$ is not preinjective, then $M$ is regular, say $M=Y_t$ for
some quasi-simple module $Y$ of rank $s$. Thus $t> 2s$ since
$|M|>|X_i|\geq 2|\delta|$, and $\mu(Y_s)\geq \mu(H_1)$. In particular,
$Y_{t-1}$ is a GR submodule of $Y_t$ and
$\mu(Y_{t-1})=\mu(X_i)<\mu(X_{i+1})<\mu(M)=\mu(Y_t)$. This is a
contradiction since $\mu(Y_t)$ is also a direct successor of
$\mu(Y_{t-1})$.
\end{proof}

\begin{prop}\label{dp} The sequence of GR measures
$$\ldots<\mu_{i+1,2}<\mu_{i+1,1}<
\mu_{i,t_i}<\ldots<\mu_{i,j+1}<\mu_{i,j}<\ldots<\mu_{i,2}<\mu_{i,1}$$
is a sequences of direct predecessors.
\end{prop}

\begin{proof}
Let $M$ be an indecomposable module such that
$$\mu(X_i)\cup\{a_{i,j+1}\}=\mu_{i,j+1}<\mu(M)<\mu_{i,j}=\mu(X_i)\cup\{a_{i,j}\}.$$
Then $\mu(M)=\mu(X_i)\cup\{b_1,b_2,\ldots,b_m\}$ with
$a_{i,j}<b_1\leq a_{i,j+1}$. By the choices of
$\mu_{i,j}$, we have $m\geq 2$ and  $b_1=a_{i,j+1}$. This implies
$M$ contains a submodule $N$ with
$\mu(N)=\mu(X_i)\cup\{a_{i,j+1}\}$, which is thus a preinjective
module by above lemma. However, an indecomposable preinjective
module can not be a submodule of any other indecomposable module.
We therefore get a contradiction.

Now let $M$ be an indecomposable module such that
$$\mu(X_i)<\mu(X_i)\cup\{|X_{i+1}|,a_{i+1,1}\}=\mu_{i+1,1}<\mu(M)<\mu_{i,t_i}=\mu(X_i)\cup\{a_{i,t_i}\}.$$
It follows that $\mu(M)=\mu(X_i)\cup\{b_1,b_2,\ldots,b_m\}$. By
definition of $\mu_{i,t_i}$, we have
$b_1=|X_{i+1}|<a_{i+1,1}<|X_{i+2}|$ and $m\geq 2$. From $b_2\leq
a_{i+1,1}$ and the definition of $\mu_{i+1,1}$, we obtain that
$b_2=a_{i+1,1}$ and $m\geq 3$. Therefore, $M$ contains an
indecomposable preinjective module $N$ with GR measure
$\mu(X_i)\cup\{|X_{i+1}|,a_{i+1,1}\}$ as a submodule, which is
impossible.

The proof is completed.
\end{proof}

\begin{remark} We should note that some segments of the sequence of
the GR measures in this proposition may not exist. In this case, we
can still show as in the proof that for example, $\mu_{j,1}$ is a
direct predecessor of $\mu_{i,t_i}$ for some $j\geq i+2$.
\end{remark}

\begin{remark} Assume that these $\mu_{i,j}$ constructed above are not
landing measures (For example,  $X$ is a homogeneous simple module
$H_1$. See Section \ref{precen}). Since each GR measure different
from $I^1$ has a direct successor, We may construct direct
successors starting from $\mu_{i,1}$ for a fixed $i$.  Let $\mu(M)$
be the direct successor of $\mu_{i,1}$. If $M$ is preinjective, then
$|M|<|\mu_{i,1}|=a_{i,1}$ by Proposition \ref{prepre}. Thus after
taking finitely many  direct successors, we obtain a  regular
measure (meaning that it is a GR measure of an indecomposable
regular module). Proposition \ref{prepre} tells that all direct
successors starting with this regular measure are still regular
ones. One the other direction, if there are infinitely many
preinjective modules containing some $X_i$, $i\geq 2r$ as GR submodules, then the
sequence $\mu_{i,j}$ is infinite (This does occur in some case. See
Section \ref{precen}). Thus we obtain a sequence of GR measures
indexed by integers $\mathbb{Z}$.

\end{remark}

\subsection{}
We fix a tame quiver $Q$ of type $\widetilde{\mathbb{A}}_n$.
There are always GR measures having no direct predecessors, for
example, $\mu(H_1)$ (Proposition \ref{H1}).  We are going to  show
that  the number of GR measures possessing no direct predecessors is
always finite.

\begin{lemm} Let $X$ be a quasi-simple module of rank $r>1$.  Assume
that there is an $i\geq 1$ such that $X_i\in\mathcal{C}$ is a
central module. Then there is an $i_0\geq i$ such that
$\mu(X_{j+1})$ is a direct successor of $\mu(X_j)$ for each $j\geq
i_0$.
\end{lemm}

\begin{proof}By Proposition \ref{ds}, we may assume that $\mu(X_r)<\mu(H_1)$.
Since $X_i$ is a central module,  $X_j$ is the unique, up to
isomorphism, GR submodule of $X_{j+1}$ for every $j\geq i$. We first
show that there is a $j_0$ such that there does not exist a regular
module with GR measure $\mu$ satisfying $\mu(X_j)<\mu<\mu(X_{j+1})$
for any $j\geq j_0$.

Let $Y$ be a quasi-simple module module of rank $s$ such that
$\mu(X_j)<\mu(Y_l)<\mu(X_{j+1})$ for some $j\geq i\geq r$ and $l\geq
1$. In this case, $Y_l$ is a GR submodule  of $Y_{l+1}$ since $Y_l$
is a central module and thus $\mu(Y_l)>\mu(T)$ for all preprojective module $T$.
Comparing the lengths, we have
$\mu(Y_{l+1})<\mu(X_{j+1})$, and similarly $\mu(Y_h)<\mu(X_{j+1})$
for all $h\geq 1$. Now replace $j$ by some $j'>j$ and repeat the
above consideration. Since there are only finitely many quasi-simple
modules $Z$ such that $\mu(Z_{r_Z})\leq \mu(H_1)$, where $r_Z$ is
the rank of $Z$, we may obtain an index $j_0$ such that a GR measure
$\mu$ of an indecomposable regular module satisfies either
$\mu<\mu(X_{j_0})$ or $\mu>\mu(X_j)$ for all $j\geq 1$.

Fix the above chosen $j_0$.  Assume that there is an indecomposable
preinjective module $M$ such that $\mu(X_j)<\mu(M)<\mu(X_{j+1})$ for
some $j\geq j_0$. Then $\mu(M)$ starts with $\mu(X_j)$ and thus
there is an indecomposable submodule $N$ of $M$ in a GR filtration
of $M$ such that $\mu(N)=\mu(X_j)$. Note that $N$ is a regular
module and thus $N=Y_l$ for some $l\geq 1$. If $X_j\cong N$, then
$\mu(M)>\mu(X_j)$ for all $j\geq 0$, a contradiction. Therefore,
$X_j\ncong N$. It follows that
$\mu(X_j)=\mu(N)<\mu(Y_{l+1})<\mu(M)<\mu(X_{j+1})$, which
contradicts the choice of $j_0$. We can finish the proof by taking
$i_0=j_0$.
\end{proof}

\begin{theo}Let $Q$ be a tame quiver of type $\widetilde{\mathbb{A}}_n$, $n\geq 1$. Then
only finitely many GR measures have no direct
predecessors.
\end{theo}
\begin{proof}
We first show that only finitely many GR measures of regular modules have no direct predecessors.
Let $X$ be a quasi-simple module of rank $r>1$. If
$\mu(X_r)\geq\mu(H_1)$, then for every $i>2r$, $\mu(X_i)$ has a
direct predecessor $\mu(X_{i-1})$ (Proposition \ref{ds}). Thus we
may assume that $\mu(X_r)<\mu(H_1)$. If every $X_i$ is a take-off
module, then $\mu(X_i)$ has direct predecessor by definition.  If
there is an index $i\geq 1$ such that $X_j$ are central modules for
all $j\geq i$, then there is an index $i_0\geq i$ such that
$\mu(X_j)$ is a direct predecessor of $\mu(X_{j+1})$ for every
$j\geq i_0$. Therefore, there are only finitely many GR measures of
indecomposable regular modules having no direct predecessor.

Now it is sufficient to show that all but
finitely many GR measures of preinjective modules have no direct
predecessors. Let $M$ be an indecomposable preinjective module.
Since there are only finitely many isomorphism classes of
indecomposable preinjective modules with length smaller than
$2|\delta|$, we may assume that $|M|>2|\delta|$. Thus a GR submodule
of $M$ is $X_i$ for some quasi-simple $X$ of rank $r\geq 1$ and some
$i\geq 2r$. Notice that $\mu(X_r)\geq \mu(H_1)$, since, otherwise,
$\mu(X_i)<\mu(H_1)<\mu(M)$ would imply $|H_1|>|M|$, which is impossible.
Without loss of generality,
we may also assume that there are GR measures $\mu$ starting with
$\mu(X_{i})$ and $\mu<\mu(M)$. (Namely, if such a $\mu$ does not
exist, we may replace $M$  by an indecomposable preinjective module
$M'$ with $|M'|>|M|+|\delta|$. Then the GR submodule of $M'$ is
$Y_{i'}$ with $Y\ncong X$. By this way, we may finally find an
integer $d$ such that all indecomposable preinjective modules with
length greater than $d$  contain $Z_l,l\geq 2r_Z$ as GR submodules
for some fixed quasi-simple module $Z$. Thus there are infinitely
many indecomposable preinjective modules with  GR measures starting
with $\mu(Z_l),l\geq 2r_Z$.) Then Proposition \ref{dp} ensures the
existence of the direct predecessor of $\mu(M)$.
\end{proof}

\subsection{Tame quivers}
After showing that  for a tame quiver of type
$\widetilde{\mathbb{A}}_n$, there are only finitely many GR measures
having no direct predecessors, we realized that the fact is also true for
any tame quiver, i.e. a quiver of type $\widetilde{\mathbb{D}}_n$,
$\widetilde{\mathbb{E}}_6$, $\widetilde{\mathbb{E}}_7$ or
$\widetilde{\mathbb{E}}_8$.  We outline the proof of this fact in the following.

\begin{theo}\label{tame}Let $\Lambda$ be a tame quiver.
Then there are only finitely many GR measures having no direct
predecessors.
\end{theo}

The proof of this theorem is almost the same as that for the
$\widetilde{\mathbb{A}}_n$ case. As we have remarked after Proposition
\ref{bigprop} that the statements (2), (4) and (7) in
Proposition \ref{bigprop} hold in general \cite{Ch4}. Using these we can show Lemma \ref{2} all tame quivers.
Proposition \ref{ds} remains true.  But the proof
should be changed a little bit because in general, a GR submodule of
a preinjective module is not necessary a regular module. The first
part of the proof of Proposition \ref{ds} is valid in general cases. For the second part, we
have to change as follows:

\begin{proof} Assume that $M$ is an indecomposable preinjective module such that
$\mu(X_i)<\mu(M)<\mu(X_{i+1})$ with $|M|$ minimal. Let $N$ be a GR
submodule of $M$. Comparing the lengths, we have
$\mu(X_i)\leq\mu(N)$. If $N=Y_j$ is regular for some quasi-simple
module $Y$ of rank $s$, then
$\mu(M)>\mu(Y_{j+1})>\mu(Y_j)\geq\mu(X_i)$. This  contradicts the
first part of the proof. If $N$ is preinjective, then
$\mu(N)=\mu(X_i)$ by the minimality of $|M|$. Thus a GR filtration
of $N$ contains a regular module $Z_{2t}$ for a quasi-simple $Z$ of
rank $t$. It follows that $\mu(X_{2r})=\mu(Z_{2t})$.  Thus by Lemma \ref{2} and Proposition \ref{bigprop}(7)
we have
$\mu(M)>\mu(N)>\mu(Z_{i+1})=\mu(X_{i+1})$, which is a contradiction.
\end{proof}

Lemma \ref{befdp} is true in general. However, Proposition \ref{dp}
should be replaced by the following one:

\begin{prop}\begin{itemize}
      \item[(1)]There are only finitely many GR measures lying between
               $\mu_{i,j}$ and $\mu_{i,j+1}$.
      \item[(2)] There are only finitely many GR measures lying between
                $\mu_{i,t_i}$ and $\mu_{i+1,1}$. In particular, $\mu_{i,j}$ has a direct predecessor.
      \end{itemize}
\end{prop}

\begin{proof} Assume that $M$ is an indecomposable module such that
$\mu(X_i)\cup\{a_{i,j+1}\}=\mu_{i,j+1}<\mu(M)<\mu_{i,j}=\mu(X_i)\cup\{a_{i,j}\}$.
Then $\mu(M)=\mu(X_i)\cup \{b_1,b_2,\ldots,b_m\}$. By definition of
$\mu_{i,j}$, we have $b_1=a_{i,j+1}$ and $m\geq 2$, In particular,
$M$ has a GR filtration containing an indecomposable module $N$ such
that $\mu(N)=\mu(X_i)\cup\{b_1\}$, which is thus preinjective.
However, there are only finitely many indecomposable modules
containing a given indecomposable preinjective module as a
submodule. It follows that only finitely many GR measures starting
with $\mu(N)=\mu(X_i)\cup\{b_1\}$.  Therefore, the number of GR
measures, which lies between $\mu_{i,j+1}$ and $\mu_{i,j}$ is finite
for each $i\geq 2r$.

2) follows similarly. Notice that
the first remark after Proposition \ref{dp} still works for this
case.
\end{proof}

The remaining proof of Theorem \ref{tame} is similar.

\section{preinjective Central modules }\label{precen}

In \cite{R1}, it was proved that all landing modules are
preinjective in the sense of Auslander and Smal\o \,\ \cite{AS}.
There may exist infinitely many preinjective central modules. In
this section, we study the preinjective modules and the
central part. Throughout this section, $Q$ is a fixed tame
quiver of type $\widetilde{\mathbb{A}}_n$.

\subsection{}We first describe the landing modules.
\begin{prop}Let $M$ be an indecomposable preinjective module.  Then
either $M\in\mathcal{L}$ or $\mu(M)<\mu(X)$ for some indecomposable
regular module $X$.
\end{prop}

\begin{proof} If $\mu(M)=\mu(X_i)$ for some quasi-simple module $X$, we thus
have $\mu(M)<\mu(X_j)$ for all $j>i$.  Thus we may assume that $\mu(M)> \mu(X)$ for all regular modules
$X\in\mathcal{R}$.  Let $\mu_1$ be the direct successor of $\mu(M)$
and $\mathcal{A}(\mu_1)$  the collection of indecomposable modules
with GR measure $\mu_1$. It follows that $\mathcal{A}(\mu_1)$
contains only preinjective modules. Let $Y^1\in\mathcal{A}(\mu_1)$
and $X^1\ra Y^1$ be a GR inclusion. Since $X^1\in\mathcal{R}$, we
have $\mu(X^1)<\mu(M)<\mu(Y^1)=\mu_1$. Thus $|M|>|Y^1|$. Let $\mu_2$
be the direct successor of $\mu_1$ and $Y^2\in\mathcal{A}(\mu_2)$.
As above we have $|Y^1|>|Y^2|$. Repeating this procedure, we get a
sequence of indecomposable preinjective modules $M=Y^0,Y^1,
Y^2,\ldots, Y^n,\ldots$ such that $\mu(Y^i)$ is the direct successor
of $\mu(Y^{i-1})$ and $|Y^i|<|Y^{i-1}|$. Because the lengths
decrease,  there is some $j<\infty$ such that $\mu(Y^j)$ has no
direct successor.  It follows that $\mu(Y^j)=I^1$ and $\mu(M)$ is a
landing measure.
\end{proof}

\begin{coro} Let $M$ be an indecomposable  module.
Then $\mu(M)>\mu(X)$ for all regular modules $X$ if and only if $M$
is a landing module.
\end{coro}

\begin{prop} If $M, N$ are landing modules, then $\mu(M)<\mu(N)$ if
and only if $|M|>|N|$.
\end{prop}

\begin{proof} Assume that $\mu(M)<\mu(N)$. Let $X$ be a GR
submodule of $N$. Since $X$ is a regular modules, we have
$\mu(X)<\mu(M)<\mu(N)$ and thus $|M|>|N|$.
\end{proof}

\begin{prop}\label{landing} Assume that there is a stable tube of rank $r>1$.  Then
all but finitely many landing modules contain only exceptional regular modules
as GR submodules.
\end{prop}

\begin{proof} Let $M$ be a landing module which is thus
preinjective. Thus the GR submodules of $M$ are all regular modules.
Assume that $M$ contains homogeneous modules $H_i$ as GR submodules.
Let $\mathbb{T}$ be a stable tube of rank $r>1$. Then there exists
a quasi-simple module $X$ on $\mathbb{T}$ such that
$\mu(X_r)\geq\mu(H_1)$ (Proposition \ref{bigprop}(5)). Thus
$\mu(H_i)<\mu(X_{r+1})<\mu(M)$ and therefore, $|X_{r+1}|>|M|$. This
implies $i=1$ and $|M|<2|\delta|$.
\end{proof}

\begin{coro}\label{cen} Assume that there is a stable tube of rank $r>1$.
If $M$ is an indecomposable module containing homogeneous modules $H_i$ as
GR submodules for some $i\geq 2$, then $M$ is a central module.
\end{coro}

\subsection{}
We partition the tame quivers of type
$\widetilde{\mathbb{A}}_n$ into three classes; for each cases we study
the possible  preinjective central modules.

{\bf Case 1} Assume that in the quiver $Q$ there is a clockwise path
of arrows $\alpha_1\alpha_2$  and a counter clockwise path
$\beta_1\beta_2$ as follows:
$$\xymatrix{\bullet\ar@{--}[d]\ar[r]^{\alpha_2} & 1 \ar[r]^{\alpha_1} & \bullet\ar@{--}[d]\\
            \bullet\ar[r]_{\beta_2} & 2 \ar[r]_{\beta_1} & \bullet} $$

Let $C$ be a string starting with $\alpha_2^{-1}$ and ending with
$\beta_2$. Thus $s(C)=1,e(C)=2$. It is obvious that the string
modules $M(C)$ contains both simple regular modules $S(1)$ and
$S(2)$, which are in different regular components,  as submodules.
Thus $M(C)$ is an indecomposable preinjective module. Fix such a
string $C$ such that the length of $C$ is large enough, i.e. $M(C)$
contains homogeneous modules $H_1$ as submodules. The GR submodules
of $M(C)$ has one of the following forms $S(1)_i, S(2)_j, H_t$ for
some $i,j,t\geq 1$. However,
$\mu(S(1)_i)<\mu(H_1),\mu(S(2)_{i})<\mu(H_1)$ for all $i\geq 0$
(Proposition \ref{bigprop}(6)). Thus the GR submodules of $M(C)$ are homogeneous
modules. In particular, there are infinitely many indecomposable
preinjective modules containing only homogeneous modules as GR
submodules.  Thus there are infinitely many preinjective central
modules by Corollary \ref{cen}.

\smallskip
As an example, we consider the following quiver
$Q=\widetilde{\mathbb{A}}_{p,q}$, $p+q=n+1$, with precisely one source and one
sink:
$$\xymatrix@R=12pt@C=18pt{
  & \bullet\ar[r]^{\alpha_2} & \bullet\ar[r]^{\alpha_3} && \cdots
                      & \bullet\ar[r]^{\alpha_{p-1}} &\bullet\ar[rd]^{\alpha_p}&\\
  1\ar[ru]^{\alpha_1}\ar[rd]_{\beta_1}&&&&&&& n+1\\
  & \bullet\ar[r]_{\beta_2} & \bullet\ar[r]_{\beta_3} & &\cdots
                      & \bullet\ar[r]_{\beta_{q-1}}&  \bullet\ar[ru]_{\beta_q} &\\ }$$

There are two stable tubes $\mathbb{T}_X$ and $\mathbb{T}_Y$
consisting of string modules. The stable tube $\mathbb{T}_Y$
contains the string module $Y$ determined by the string
$\beta_q\beta_{q-1}\cdots\beta_2\beta_1$, and simple modules $S$
corresponding to the vertices ${s(\alpha_i)}, 2\leq i\leq p$ as
quasi-simple modules. The rank of $Y$ is $p$. The other tube $\mathbb{T}_X$
contains string module $X$ determined by the string
$\alpha_p\alpha_{p-1}\cdots\alpha_2\alpha_1$ and simple modules $S$
corresponding to the vertices ${s(\beta_i)}, 2\leq i\leq q$. The
ranks of $X$ is $q$. All the other stable tubes contain only  band
modules.

We can easily determine the GR measures of these quasi-simple
modules. Notice that any non-simple quasi-simple module ($X$,$Y$ and
$H_1$) contains $S(n+1)$ as the unique simple submodule. Therefore,
each homogeneous simple module $H_1$ has GR measure
$\mu(H_1)=\{1,2,3,\ldots, n,n+1\}$ and the GR measure for $X$ and
$Y$ are $\mu(X)=\{1,2,3,\ldots,p,p+1\}$ and
$\mu(Y)=\{1,2,3,\ldots,q,q+1\}$. It is easily seen that
 $X_q\subset X_{q+1}\subset\ldots
\subset X_j\subset\ldots $ is a chain of GR inclusions and thus
$\mu(X_q)=\mu(H_1)$. Similarly,  $Y_p\subset Y_{p+1}\subset\ldots
\subset Y_j\subset\ldots $ is a chain of GR inclusions and
$\mu(Y_p)=\mu(H_1)$.

Any non-sincere indecomposable module belongs to the take-off part. This
is true because the GR submodule of $H_1$ is a uniserial module and
has GR measure $\{1,2,3,\ldots,n\}$ and a non-sincere
indecomposable module has length smaller than $|\delta|$.  Let $M\in
\mathcal{I}$ be a sincere indecomposable preinjective module and
$T\subset M$ a GR submodule. We claim that $T$ is isomorphic to some $H_i$,
$X_{sq}$ or $Y_{tp}$ for some $i, s,t\geq 1$. First of all the $T\ncong S_i$
for any simple regular module $S$ and any $i\geq 1$, since $\mu(S_i)<\mu(H_1)$ (Proposition \ref{bigprop}(6))
and there is monomorphism $H_1\ra M$ for each homogeneous module $H_1$.
If, for example, $T\cong X_i$ for some $i\geq q$, then there is an epimorphism $X_{i+1}\ra M$ (Proposition \ref{epi}(2)). Thus $|X_i|<|M|<|X_{i+1}|$. However, $|X_{i+1}|-|X_i|=1$ if $i$ is not divided by $q$.

Notice that if $p\geq 2$ and $q\geq 2$, then there are infinitely
many preinjective central modules by above discussion.

\smallskip

{\bf Case 2}  $Q=\widetilde{\mathbb{A}}_{p,1}$. Let's keep the
notations in the above example. By Proposition \ref{landing}, we
know that there are infinitely many landing modules containing only
exceptional modules of the form $Y_i$  as GR submodules. Given an
indecomposable preinjective module $M$ and its GR submodule $Y_i, i>
p$. We claim that the GR submodules of $\tau M$ are homogeneous
ones. Namely, if $\tau M$ contains an exceptional regular module $N$
as a GR submodule, then $N\cong Y_j$ for some $j\geq p$. In
particular, both $M$ and $\tau M$ contains $Y$ as a submodule, i.e.
$\Hom(Y,M)\neq 0\neq\Hom(Y,\tau M)$. Therefore, we have
$\Hom(\tau^{-}Y, M)\neq 0\neq \Hom(Y,M)$, which contradicts Lemma
\ref{onemap}.  Thus, there are infinite many indecomposable
preinjective modules containing only homogeneous modules as GR
submodules and hence infinitely many preinjective central modules.

\smallskip

{\bf Case 3} $Q\neq \widetilde{\mathbb{A}}_{p,q}$ is of the
following form: all non-trivial clockwise  (or counter clockwise)
paths (compositions of arrows) are of length $1$. In this case, all
exceptional quasi-simple modules in one of the exceptional tubes are
of length at least $2$, and the quasi-simple modules on the other
exceptional tube have length at most $2$.

Let $p=\beta_t\ldots\beta_2\beta_1$ be a composition of arrows in
$Q$ with maximal length. Thus there is an arrow $\alpha$ with ending
vertex $e(\alpha)=e(p)$ and $s(\alpha)$ is a source. Let $X=M(p)$ be
the string module, which is thus a quasi-simple module, say with
rank $r$. By the maximality of $p$ and the description of
irreducible maps between string modules, we may easily deduce that
the sequence of irreducible monomorphism $X=X_1\ra X_2\ra \ldots \ra
X_r\ra X_{r+1}\ra \ldots $ is namely a sequence of GR inclusions.
Therefore
$$\mu(X_{r+1})=\{1,2,3,\ldots,t+1=|X_1|, |X_2|, |X_3|,\ldots, |X_r|,
|X_{r+1}|\}$$ with $|X_i|-|X_{i-1}|\geq 2$ for $2\leq i\leq r$ and
$|X_{r+1}|=|X_r|+(t+1)$.

Let $Y$ be the string module determined by the arrow $\alpha$. It is
also a  quasi-simple module, say with rank $s$. It is clear that $X$ and $Y$ are
not in the same regular component. By the
description  of irreducible monomorphisms, we obtain that $|Y_j|=j+1$
for $j\leq t$ and $|Y_{t+1}|=t+3$. Thus
$$\mu(Y_{s+1})\geq \{1,2,\ldots, t+1,t+3, |Y_{t+2}|,\ldots,
|Y_s|,|Y_{s+1}|\}$$ with $|Y_i|-|Y_{i-1}|\leq 2$ for $i\leq s$ and
$|Y_{s+1}|=|Y_s|+2$.

This proves the following lemma.
\begin{lemm}
We keep the notations as above. If $t>1$, i.e. $Q$ is not equipped
with a sink-source orientation, then $\mu(Y_s) \geq \mu(X_r) $ and
$\mu(Y_j)>\mu(X_i)$ for $i\geq 1$ and $j> s$.
\end{lemm}

\subsection{}
Now we characterize the tame quivers $Q$ of type
$\widetilde{\mathbb{A}}_n$ such that no indecomposable preinjective
modules are central modules. We also show that there are always
infinitely many preinjective central modules if any.

\begin{theo}Let $Q$ be a tame quiver of type $\widetilde{\mathbb{A}}_n$.
Then $\mathcal{I}\cap\mathcal{C}=\emptyset$ if and only if $Q$ is equipped with
a sink-source orientation.
\end{theo}
\begin{proof}
If $n=1$, then $Q$ is obvious of a sink-source orientation, and the central part
contains precisely the regular modules (see, for example, \cite{R1}). Now we assume
$n\geq 2$. Thus there exists an exceptional regular component.  If
$\mathcal{I}\cap\mathcal{C}=\emptyset$, a sincere indecomposable
preinjective is always a landing module. Then the proof of
Proposition \ref{landing} implies that there is no indecomposable
preinjective modules $M$ containing  a homogeneous module $H_i,i\geq
2$ as GR submodules. Therefore, by above partition of the tame quiver of type $\widetilde{\mathbb{A}}_n$,  we
need only to consider Case 3 and show that
$\mathcal{I}\cap\mathcal{C}=\emptyset$ implies $t=1$ (let's keep the
notations in case 3). Assume for a contradiction that $t>1$. Let $S$
be the simple module corresponding to $s(\beta_t)$. Thus $S$ is a
quasi-simple of rank $s$ and $ \tau S\cong Y$. Let $I$ be the
(indecomposable) injective cover of $S$. It is obvious that
$\Hom(X,I)\neq 0$. Consider the indecomposable preinjective modules
$\tau^{um}I$, where $u$ is a positive integer and  $m=[r,s]$ is the
lowest common multiple of $r$ and $s$.  Since
$\Hom(S,\tau^{um}I)\neq 0\neq \Hom(X,\tau^{um}I)$, a GR submodule of
$\tau^{um}I$ is either $S_i$ or $X_j$. Notice that
$\mu(H_1)>\mu(S_i)$ for all $i\geq 0$ since $S$ is simple.
Therefore, for $u$ large enough, the unique GR submodule of
$\tau^{um}I$ is $X_j$ for some $j\geq 1$ because no indecomposable
preinjective modules containing $H_i$ as GR submodules for $i\geq
2$. In particular there are infinitely many preinjective modules
containing GR submodules of the form $X_j,j\geq 1$. Thus we may
select a GR inclusion $X_j\subset M$ with $M\in\mathcal{I}$ such $|X_j|>|Y_{s+1}|$. Because
$\mu(X_j)<\mu(Y_{s+1})<\mu(M)$, we have $|Y_{s+1}|>|M|$.  This
contradicts $|X_j|>|Y_{s+1}|$. Thus we have $t=1$ and $Q$ is equipped with a
sink-source orientation.

Conversely, if  $Q$ is with a
sink-source orientation, we may see directly that
$\mathcal{I}\cap\mathcal{C}=\emptyset$ (for details, see \cite[Example 3]{Ch3}).
\end{proof}

\begin{theo}Let $Q$ be a tame quiver of type $\widetilde{\mathbb{A}}_n$. Then $\mathcal{I}\cap\mathcal{C}\neq
\emptyset$ if and only if $|\mathcal{I}\cap\mathcal{C}|=\infty$.
\end{theo}

\begin{proof} We have seen in Corollary \ref{cen} that an indecomposable module containing homogeneous modules
$H_i, i\geq 2$ as GR submodules is a central module. Thus we may
assume that there are only finitely many indecomposable preinjective
module containing homogenous modules as GR submodules. Thus, we need
only consider Case 3. Let's keep the notations there. Then
$\mathcal{I}\cap\mathcal{C}\neq \emptyset$ implies that $Q$ is not
with a sink-source orientation. In particular, the length $t$ of the
longest path of arrows $\beta_t\cdots\beta_1$ is greater than $1$.
Therefore,  $\mu(Y_j)>\mu(X_i)$ for all $i\geq 1, j>s$. Again let
$m=[r,s]$. By assumption, the GR submodules of $\tau^{um}I$ are of
the form $X_i$ for almost all $u\geq 1$. To avoid a contradiction as
in the proof of above theorem, $\mu(\tau^{um}I)$ is smaller
than $\mu(Y_{s+1})$ for $u$ large enough and thus almost all
$\tau^{um}I$ are central modules.
\end{proof}
\bigskip

{\bf Acknowledgments.} The author is grateful to the referees for valuable
comments and helpful suggestions, which make the article more
readable.  He also wants to thank Jan Schr\"oer for many useful discussions.

\end{document}